\documentclass[APA,LATO1COL]{WileyNJD-v2}
\usepackage{amsmath}
\usepackage{amsthm}
\usepackage{amsfonts}
\usepackage{dsfont}
\usepackage{multicol}
\usepackage{enumerate}
\articletype{Article Type}%
\received{26 April 2016}
\revised{6 June 2016}
\accepted{6 June 2016}

\begin{document}
\title{Asymptotic tail properties of Poisson mixture distributions}

\author[1,2,3,4]{Samuel Valiquette}

\author[2,3]{Gwladys Toulemonde}

\author[2]{Jean Peyhardi}

\author[4]{Éric Marchand}

\author[1,5]{Frédéric Mortier}

\authormark{Valiquette \textsc{et al.}}

\address[1]{\orgdiv{UPR Forêt et Sociétés}, \orgname{CIRAD}, \orgaddress{F-34398 Montpellier, France.
Forêts et Sociétés, Univ Montpellier, CIRAD, Montpellier, France.}}

\address[2]{\orgdiv{IMAG, CNRS}, \orgname{Université de Montpellier}, \orgaddress{34090, Montpellier, France}}

\address[3]{\orgdiv{LEMON}, \orgname{Inria}, \orgaddress{34095, Montpellier, France}}

\address[4]{\orgdiv{Département de mathématiques}, \orgname{Université de Sherbrooke}, \orgaddress{Sherbrooke, Canada, J1K 2R1}}
\address[5]{\orgdiv{Environmental Justice Program}, \orgname{Georgetown University, Washington D.C.}, \orgaddress{ United States of America}}

\corres{Samuel Valiquette, Département de mathématiques, Université de Sherbrooke, Sherbrooke, Canada, J1K 2R1.\\ \email{samuel.valiquette@usherbrooke.ca}}

\abstract[Summary]{Count data are omnipresent in many applied fields, often with overdispersion.  
With mixtures of Poisson distributions representing an elegant and appealing modelling strategy, we focus here on how the tail behaviour of the mixing distribution is related to the tail of the resulting Poisson mixture.
We define five sets of mixing distributions and we identify for each case whenever the Poisson mixture is in, close to or far from a domain of attraction of maxima.
We also characterize how the Poisson mixture behaves similarly to a standard Poisson distribution when the mixing distribution has a finite support.
Finally, we study, both analytically and numerically, how goodness-of-fit can be assessed with the inspection of tail behaviour.}


\maketitle

\section{Introduction}
Count data are classically observed in many applied fields such as in actuarial science when evaluating risk and the pricing of insurance contracts \citep[e.g.,][]{Claims}, in genetics to model the number of genes involved in phenotype variability \citep[e.g.,][]{Genes} or in ecology to model species abundance  \citep[e.g.,][]{Abundance}.
While Poisson models and regression are well established choices for these type of data, they are not suitable for overdispersed data.  
To overcome such limitations the use of Poisson mixture models has been proposed.
This assumes the Poisson's intensity is no longer an unknown fixed value, but a  positive random variable. 
A variety of mixture distributions has been already proposed \citep{Karlis} and classical examples includes the gamma distribution \citep{Greenwood}, the lognormal \citep{Bulmer} or the Bernoulli \citep{Lambert}.
As demonstrated by \cite{Feller}, Poisson mixtures are uniquely identifiable by the mixing distribution on the Poisson parameter $\lambda$.
Therefore, it suffices to take into account the behaviour of the mixing distribution when it comes to adjusting count data with a Poisson mixture model. In particular, the mixing distribution should reflect the tail behaviour of the count data.

The field of extreme value theory allows to analyze such a behaviour through the distribution of maxima.
Precisely, the tail behaviour of a random variable can be characterizes by three domains of attraction \citep{Resnick}: Weibull, Gumbel and Fréchet.
Most familiar continuous distributions can be associated to one of these domain of attraction.
For discrete distributions, \cite{Anderson} identified three different cases.
A sample drawn from a discrete distribution is either: (i) in a domain of attraction, (ii) "close" to the Gumbel domain of attraction or (iii) drastically fails to belong to one such that their maxima oscillates between two increasing integers as the sample size grows to infinity.
\cite{Perline} provided conditions on the mixing distribution such that the Poisson mixture remains in the Fréchet or Gumbel domain of attraction, i.e. case (i).
However, they did not investigate what type of distributions on $\lambda$ causes the Poisson mixture to satisfy case (ii) or (iii). 
This article aims to complete their work by identifying what conditions on the mixing distribution allow the Poisson mixture to be associated to the two latter cases. 
Moreover, we demonstrate that their condition for the Fréchet domain of attraction is not necessary in order for the Poisson mixture to remain in this domain.

This paper is organized as follows.
Section \ref{Sect:Mixture} presents the extreme value theory in the Poisson mixture context and different families of mixing distributions.
Using these set of distributions, we identify when the Poisson mixture is in, close to, or far from a domain of attraction.
Moreover, we demonstrate that Poisson mixtures satisfying the latter case behave similarly to a standard Poisson distribution.
In Section \ref{Section:Numerical}, we inspect how those three situations can affect the goodness-of-fit when it comes to adjusting count data with a Poisson mixture.
Moreover, we explore how one can identify which type of mixing distribution can be adequate by using the generalized Pareto distribution on the excesses.
We also study how the closeness to the Gumbel domain of attraction has an impact on identifying such a mixing distribution.
Finally, we provide an example where the maxima of a Poisson mixture alternates between two values.

\section{Poisson mixture tail behaviour}\label{Sect:Mixture}

In this section, we present notations and the family of mixing distributions that is studied in this paper.
Moreover, preliminary results in extreme value theory are presented and we describe maximum domain of attraction restrictions for discrete distributions.
Following this, we elaborate on mixing distributions that allow the Poisson mixture to be either in or near a domain of attraction, or to drastically fail to belong in one.
Finally, for a Poisson mixture with a finite mixing distribution, we will prove that the asymptotic behaviour of its probability mass function behaves similarly to that of a Poisson distribution. 

\subsection{Theoretical foundations}

In the following, for Poisson mixtures $X\mid \lambda$ with $\lambda$ random, $F$, $\overline{F}$ and $f$ will denote respectively the cumulative distribution function (cdf), the survival function, and the probability density function (pdf) for the mixing $\lambda$. 
Similarly, $F_M$, $\overline{F}_M$ and $P_M$ will denote respectively the cdf, the survival function, and the probability mass function (pmf) of the resulting Poisson mixture $X$.
Moreover, in this paper, we restrict the mixing distributions on $\lambda$ to those with a support equal to $(0, x_0)$ for $x_0 \in \mathbb{R}_+ \cup \{\infty\}$.
Finally, we require the notion of a slowly varying function $C(x)$ on $\mathbb{R}_+$, defined by the property: for every $t \in \mathbb{R}_+$, $C(tx)\sim C(x)$, where $g(x) \sim h(x)$ means that $\lim_{x\to \infty} \frac{g(x)}{h(x)} = 1$ for functions $g$ and $h$.\medskip

The tail behaviour of the Poisson mixture can be studied using extreme value theory.
Such a statistical approach analyzes how the maxima of $F_M$ stabilizes asymptotically.
For a general distribution $G$, the theory says that $G$ belongs to a domain of attraction if there exist two normalizing sequences $a_n > 0$ and $b_n$ such that $G^n(a_n x + b_n)$ converges to a non-degenerate distribution when $n$ tends to infinity \citep{Resnick}.
Such a non-degenerate distribution can only be the generalized extreme value distribution given by
\begin{equation}
\label{Eqn:GEV}
  \lim_{n \to \infty} G^n(a_n x + b_n) =
    \begin{cases}
      \exp \left[ - (1 + \gamma x)^{-1/\gamma}\right] & \text{for $1 + \gamma x > 0$ with $\gamma \neq 0$;}\\
      \exp \left[ - e^{-x} \right] & \text{for $x \in \mathbb{R}$ with $\gamma = 0$.}
    \end{cases}       
\end{equation}
The three possible domains of attraction are named Weibull, Gumbel and Fréchet for $\gamma < 0$, $\gamma = 0$ and $\gamma > 0$ respectively, and will be denoted by $\mathcal{D}_-$, $\mathcal{D}_0$ and $\mathcal{D}_+$. 
Accordingly, we will write $G \in \mathcal{D}$ where $\mathcal{D}$ is one of the three domains.
Necessary and sufficient conditions for $G$ to be in a domain of attraction have been established by \citet{Gnedenko}.
While most common continuous distributions can be associated to a domain of attraction, this is not always the case for discrete random variables.
Indeed, a necessary condition for a discrete distribution $G$ to be in a domain of attraction is the long-tailed property \citep{Anderson} defined by
\begin{eqnarray}\label{Eqn:LongTail}
    \overline{G}(n+1) \sim \overline{G}(n).
\end{eqnarray}
Well known discrete distributions, such as Poisson, geometric and negative binomial, do not satisfy the above property.
However, \citet{Anderson} and \citet{Shimura} showed that if a discrete distribution verifies 
\begin{eqnarray}\label{Eqn:LightTail}
\overline{G}(n+1) \sim L \overline{G}(n),
\end{eqnarray}
for $L \in (0,1)$, then $G$ is, in a sense, "close" to the Gumbel domain.
More precisely, \citet{Shimura} showed that property \eqref{Eqn:LightTail} implies that $G$ is the discretization of a unique continuous distribution belonging to $\mathcal{D}_0$.
On the other hand, \cite{Anderson} showed that there exist a sequence $b_n$ and $\alpha > 0$ such that 
\begin{align*}
    \limsup_{n \to \infty} G^n(x+b_n) &\leq \exp\left(-e^{-\alpha x}\right)\\
    \liminf_{n \to \infty} G^n(x+b_n) &\geq \exp\left(-e^{-\alpha (x-1)}\right)
\end{align*}
if and only if $\overline{G}(n+1) \sim e^{-\alpha}\overline{G}(n).$
Therefore, the supremum and infimum limits of $G^n(x+b_n)$ are bounded by two Gumbel distributions under condition \eqref{Eqn:LightTail}.
The geometric and negative binomial distributions are two such examples. 
Finally, if the discrete distribution is a Poisson, or more generally such that
\begin{eqnarray}\label{Eqn:ShortTail}
\lim_{n \to \infty} \frac{\overline{G}(n+1)}{\overline{G}(n)} = 0,
\end{eqnarray}
then no sequence $b_n$ can be found such that the the supremum and infimum limits of $G^n(x+b_n)$ are bounded by two different Gumbel distributions.
For this case, \cite{Anderson} showed that for $Y_i \overset{iid}{\sim} G$, there exists a sequence of integers $I_n$ such that
\begin{eqnarray}\label{Eqn:MaxSeq}
\lim_{n \to \infty} P\left(\max_{1\leq i \leq n} Y_i = I_n \text{ or } I_n + 1\right) =1
\end{eqnarray}
if and only if \eqref{Eqn:ShortTail} is satisfied.
Therefore, the maximum of such discrete distribution oscillates between two integers asymptotically.

\subsection{Poisson mixtures categories}

Since Poisson mixture distributions are discrete distributions, they are constrained to the long-tailed property \eqref{Eqn:LongTail} in order to have a domain of attraction. 
Otherwise, they may be close to the Gumbel domain or with a maximum alternating between two integers.
Since a Poisson mixture is uniquely identifiable by the distribution on $\lambda$ \citep{Feller}, it follows that its tail behaviour depends on the latter.
Therefore, we seek to identify what conditions on the distribution of $\lambda$ allow the Poisson mixture distributions to satisfy either equation \eqref{Eqn:LongTail}, \eqref{Eqn:LightTail} or \eqref{Eqn:ShortTail}.
In the following, we will establish that Poisson mixtures with $F$ in $\mathcal{D}_+$ or $\mathcal{D}_-$ will satisfy equations \eqref{Eqn:LongTail} and \eqref{Eqn:ShortTail} respectively, but for mixing distributions in $\mathcal{D}_0$, the Poisson mixture may satisfy either one of the three limits depending on their behaviour.
We require the following definitions and notations.
\begin{definition}
\label{Def:ExpTail}
A distribution $F$ has an \textbf{exponential tail} if for all $k \in \mathbb{R}$, there is a $\beta > 0$ such that for $x\to \infty$
\begin{equation}\label{Eqn:exptail}
    \overline{F}(x+k) \sim e^{-\beta k} \overline{F}(x).
\end{equation} 
\end{definition}

\begin{definition}
\label{Def:GumbelCond}
    A distribution $F$ satisfies the \textbf{Gumbel hazard condition} if its density $f$ has a negative derivative for all $x$ in some left neighborhood of $\{\infty\}$, $\lim_{x \to \infty} \frac{d}{dx}\left[ \frac{1-F(x)}{f(x)}  \right] = 0$ (the 3rd Von Mises Condition) and $\lim_{x \to \infty} \frac{x^{\delta}f(x)}{1-F(x)} = 0$ for some $\delta \geq \frac{1}{2}$.
\end{definition}
\noindent
Using Definitions \ref{Def:ExpTail} and \ref{Def:GumbelCond}, we focus on three distinct subsets of $\mathcal{D}_0$.
Firstly, distributions satisfying one of these definitions are in the Gumbel domain of attraction, see \cite{Shimura} and \cite{Resnick}.
Secondly, some distributions with finite tail are in $\mathcal{D}_0$, but do not belong to $\mathcal{D}_-$ (e.g. \cite{Gnedenko}).
Based on these three cases, let $\mathcal{D}_0^\mathcal{E}$, $\mathcal{D}_0^\mathcal{H}$ and $\mathcal{D}_0^\mathcal{F}$ denote respectively the classes of $F \in \mathcal{D}_0$ satisfying Definition \ref{Def:ExpTail}, Definition \ref{Def:GumbelCond}, and with finite tail.
These subsets of $\mathcal{D}_0$ are disjoint by the following Proposition.
\begin{proposition}
\label{Prop:exptail}
The sets $\mathcal{D}_0^\mathcal{E}$, $\mathcal{D}_0^\mathcal{H}$ and $\mathcal{D}_0^\mathcal{F}$ are disjoint.
\end{proposition}
\begin{proof}
Since $\mathcal{D}_0^\mathcal{E}$ and $\mathcal{D}_0^\mathcal{H}$ represent distributions with an infinite tail, they are both disjoint from $\mathcal{D}_0^\mathcal{F}$.
To establish that $\mathcal{D}_0^\mathcal{E}$ and $\mathcal{D}_0^\mathcal{H}$ are disjoint, it is sufficient to show that the condition on the hazard rate function in Definition \ref{Def:GumbelCond} is not satisfied for the former case.
By abuse of notation, we will denote in this proof by $C$ any slowly varying function.
As noticed in \cite{Cline}, $F$ has an exponential tail if and only if $\overline{F}(\ln x) = C(x) x^{-\beta}$ for some $\beta > 0$.
By the monotone density theorem presented in Theorem 1.7.2. in \cite{Bingham}, we then have $f(x) = C(e^x)e^{-\beta x}$.
Therefore, the density $f$ still has an exponential tail.
Moreover, its limit given by Definition \ref{Def:ExpTail} converges uniformly on $(\ln b, \infty)$ for every $b > 0$ \citep{Resnick}.
Then, by differentiating $f$ with respect to $k$,
$$\lim_{x \to \infty} \frac{f'(x+k)}{f(x)} = - \beta e^{-\beta k },$$
and fixing $k = 0$, we conclude that $f'(x) \sim -\beta f(x)$.
Using this property, one has for all $\delta  > 0$ that
$$\lim_{x \to \infty} \frac{x^\delta f(x)}{\overline{F}(x)} = \beta \lim_{x \to \infty} x^\delta = \infty,$$
showing that the Gumbel hazard condition (Definition \ref{Def:GumbelCond}) is not satisfied.
\end{proof}
\noindent
Although these subsets are disjoint, they do not form a partition of $\mathcal{D}_0$.
Indeed, the Weibull distribution with cdf $F(x) = 1 - e^{-\left(\frac{x}{\beta}\right)^\alpha}$ is neither in $\mathcal{D}_0^\mathcal{H}$ or $\mathcal{D}_0^\mathcal{E}$ when $\alpha \not\in (0,1/2) \cup \{1\}$.
This distribution belongs to a broader subset of $\mathcal{D}_0$ named Weibull tail which intersects with $\mathcal{D}_0^\mathcal{H}$ and $\mathcal{D}_0^\mathcal{E}$; see \cite{Gardes} for more details.
We now discriminate between properties \eqref{Eqn:LongTail}, \eqref{Eqn:LightTail} or \eqref{Eqn:ShortTail} with respect to the domain of attraction of $\lambda$.
\begin{theorem}
Let $F_M$ be a Poisson mixture with $\lambda$ distributed according to a cdf $F$ and supported on $(0,x_0)$ with $x_0\in\mathbb{R}_+ \cup \{\infty\}$.
Then for any integer $k \geq 1$,
\begin{equation*}
  \lim_{n \to \infty} \frac{\overline{F}_M(n+k)}{\overline{F}_M(n)} =
    \begin{cases}
      1 & \text{if $F \in \mathcal{D}_+ \cup \mathcal{D}_0^\mathcal{H}$,}\\
      (1 + \beta)^{-k} & \text{if $F \in \mathcal{D}_0^\mathcal{E}$,}\\
      0 & \text{if $F \in \mathcal{D}_- \cup \mathcal{D}_0^\mathcal{F}$,}
    \end{cases}       
\end{equation*}
where $\beta > 0$ is given by Definition \ref{Def:ExpTail} for $\mathcal{D}_0^\mathcal{E}$.
\label{Thm:Ratio}
\end{theorem}
\begin{proof}
    (A) $ \lim_{n \to \infty} \frac{\overline{F}_M(n+k)}{\overline{F}_M(n)} = 1$: The result for $\mathcal{D}_0^\mathcal{H}$ is directly established by \cite{Perline}.
    For $F \in \mathcal{D}_+$, a necessary and sufficient condition is that $\overline{F}(x) = C(x)x^{-\alpha}$ with $\alpha > 0$ \citep{Gnedenko}.
    In fact, $C(x)$ must be locally bounded since $\overline{F}$ is bounded.
    As presented in \cite{Karlis}, the survival function of the mixture is given by
    $$\overline{F}_M(x) = \int_0^\infty \frac{\lambda^{\lfloor x \rfloor} e^{-\lambda}}{\lfloor x \rfloor !} (1-F(\lambda)) d\lambda = \frac{\Gamma(\lfloor x \rfloor - \alpha + 1)}{\Gamma(\lfloor x \rfloor + 1)} \int_0^\infty \underbrace{\frac{\lambda^{\lfloor x \rfloor - \alpha} e^{-\lambda}}{\Gamma(\lfloor x \rfloor - \alpha + 1)}}_{g(x,\lambda)} C(\lambda) d\lambda$$
    for $x$ such that $\lfloor x \rfloor - \alpha > 0$.
    By the definition of the Gamma function, $\int_0^\infty g(x, \lambda) d\lambda = 1$, and then for $0 \leq a < b \leq \infty$, $\phi \in (-1,1)$, and Stirling's formula, we have
    $$\int_a^b \lambda^\phi g(x, \lambda) d\lambda \leq \int_0^\infty \lambda^\phi g(x, \lambda) d\lambda = \frac{\Gamma(\lfloor x \rfloor - \alpha + \phi + 1)}{\Gamma(\lfloor x \rfloor - \alpha + 1)} \sim \lfloor x \rfloor^\phi.$$
    By Theorem 4.1.4 in \cite{Bingham}, we can conclude that $\overline{F}_M$ is such that
    $$\overline{F}_M(x) \sim C(\lfloor x \rfloor)  \frac{\Gamma(\lfloor x \rfloor - \alpha + 1)}{\Gamma(\lfloor x \rfloor + 1)} \sim C(\lfloor x \rfloor) \lfloor x \rfloor^{-\alpha}.$$
    Furthermore, since $\lfloor x \rfloor \sim x$, $C(\lfloor x \rfloor) \sim C(x)$ using the Karamata representation of $C$ \citep{Resnick}.
    Therefore $F_M \in \mathcal{D}_+$ and $\overline{F}_M(n+k) \sim \overline{F}_M(n)$.\medskip
    
    \noindent (B) $ \lim_{n \to \infty} \frac{\overline{F}_M(n+k)}{\overline{F}_M(n)} = (1+\beta)^{-k}$: Since $F$ has an exponential tail, then $\overline{F}(x) = C(e^x) e^{-\beta x}$ for some $\beta > 0$.
    Using a similar argument as in Theorem 4.1.4 in \cite{Bingham}, we can prove that 
    $$\overline{F}_M(n) \sim \frac{C(e^n)}{(1+\beta)^{n+1}}.$$
    Therefore, 
    $$\lim_{n \to \infty} \frac{1-F_M(n+k)}{1-F_M(n)} =  (1+\beta)^{-k} \lim_{n \to \infty} \frac{C(e^{n+k})}{C(e^n)} = (1+\beta)^{-k}.$$\medskip
    
    \noindent (C) $ \lim_{n \to \infty} \frac{\overline{F}_M(n+k)}{\overline{F}_M(n)} = 0$: Because $\overline{F}_M(n) = \int_0^{x_0} \frac{\lambda^n e^{-\lambda}}{n!} (1-F(\lambda)) d\lambda$, the result as above follows since
    \begin{align*}
        \frac{\overline{F}_M(n+k)}{\overline{F}_M(n)} &= \frac{1}{\prod_{i=1}^k (n+i)} \frac{\int_0^{x_0} \lambda^{n+k} e^{-\lambda} (1-F(\lambda)) d\lambda}{\int_0^{x_0} \lambda^{n} e^{-\lambda} (1-F(\lambda)) d\lambda}\\
        &\leq \frac{x_0^k}{\prod_{i=1}^k (n+i)} \to 0 \text{ when $n \to \infty$}.
    \end{align*}
\end{proof}
Theorem \ref{Thm:Ratio} establishes that if $F \in \mathcal{D}_+$, then $F_M \in \mathcal{D}_+$ which improves the result of \cite{Perline} that adds the 1st Von Mises condition \citep{Resnick} to proved a similar result.
By relaxing such a condition, we proved that any mixing distributions in $\mathcal{D}_+$ allows the Poisson mixture to remain in this domain of attraction.
Analogous to this property, \cite{Shimura} showed that any discretization of a continuous distribution in $\mathcal{D}_+$ preserves the domain of attraction.
Considering the Poisson mixture as a discretization operator, we obtain another example where the Fréchet domain of attraction is preserved.
A broad set of mixing distributions in $\mathcal{D}_+$ can be found, for example the Fréchet, folded-Cauchy, Beta type II, inverse-Gamma, or the Gamma/Beta type II mixture \citep{Irwin}.
Unfortunately, examples are scarce for distributions in $\mathcal{D}_0^\mathcal{H}$.
Indeed the asymptotic behaviour of the hazard rate function in Definition \ref{Def:GumbelCond} is quite restrictive.
Examples include the lognormal, the Benktander type I and II \citep{Benktander}, and the Weibull distributions, with further restrictions on the parameters for the latter two cases.
These type of distributions do not encompass cases like the Gamma, even though the associated mixing distribution belongs to $\mathcal{D}_0$, because it does not satisfy the additional condition on the hazard rate function.
The class $\mathcal{D}_0^\mathcal{E}$ allows to describe such mixing distribution.
It includes a broad class of elements among others Gamma, Gamma/Gompertz, exponential, exponential logarithmic, inverse-Gaussian and the generalized inverse-Gaussian.
As previously mentioned these distributions are in the Gumbel domain of attraction, but from Theorem \ref{Thm:Ratio}, the resulting Poisson mixtures do not belong to any domain of attraction.
However, we can quantify how close such Poisson mixtures are to the Gumbel domain of attraction.
Indeed, if $\beta \to 0$ then $\frac{1-F_M(n+1)}{1-F_M(n)} \to 1$, i.e. it approaches a long-tailed distribution.
Finally, when $F$ has a finite tail, i.e. $F \in \mathcal{D}_- \cup \mathcal{D}_0^\mathcal{F}$, the Poisson mixture cannot be close to any domain of attraction by Theorem \ref{Thm:Ratio}.

\subsection{Asymptotic behaviour for $F \in \mathcal{D}_-$}
To shed light on why the last limit in Theorem \ref{Thm:Ratio} is null, we complete this section by studying the asymptotic behaviour of the pmf $P_M$ when $F$ is in $\mathcal{D}_-$.
\cite{Willmot} studied such a behaviour when the Poisson mixture has a mixing distribution with a particular exponential tail.
This result is presented in the following Proposition.
\begin{proposition}[\cite{Willmot}]
\label{Prop:Willmot}
Let $F_M$ be a Poisson mixture with $\lambda$ distributed according to a distribution $F$ such that its density is $$f(x) \sim C(x) x^{\alpha} e^{-\beta x},$$
where $C$ is a locally bounded and slowly varying function on $\mathbb{R}_+$, and for some $\alpha \in \mathbb{R}$ and $\beta > 0$.
Then the pmf $P_M$ is such that
$$P_M(n) \sim C(n) n^\alpha (1+\beta)^{-(n+\alpha +1)}.$$
\end{proposition}
Proposition \ref{Prop:Willmot} indicates that when the density $f$ behaves similarly to a Gamma distribution, then the pmf $P_M$ behaves like a negative binomial pmf multiplied by a regular varying function.
As previously mentioned, the negative binomial is an example of a distribution where equation \eqref{Eqn:LightTail} is satisfied.
This provides additional clarification on why the limit associated with an exponential tail in Theorem \ref{Thm:Ratio} converges to a value between $0$ and $1$.
In the following Theorem, a similar conclusion is presented when $F \in \mathcal{D}_-$.
\begin{theorem}
\label{Thm:Asymp}
Let $F_M$ be a Poisson mixture with $\lambda$ distributed according to a distribution $F \in \mathcal{D}_-$.
Then there exists an $\alpha > 0$ such that
$$\overline{F}_M(n) \sim \Gamma(\alpha + 1) C(n) n^{-\alpha} \left( \frac{x_0^{n+1}}{(n+1)!} e^{-x_0} \right).$$    
\end{theorem}
\begin{proof}
    Using the integral representation of $\overline{F}_M$, 
    $$\overline{F}_M(n) = \int_0^{x_0} \frac{\lambda^n e^{-\lambda}}{n!}(1-F(\lambda))d\lambda = \frac{x_0^{n+1}}{n!} \int_0^{\infty} \frac{\lambda^n}{(\lambda+1)^{n+2}} e^{-\frac{x_0 \lambda}{\lambda + 1}} \left(1-F\left( \frac{x_0 \lambda}{\lambda + 1}\right) \right)d\lambda$$
    where the transformation $\lambda \mapsto \frac{\lambda}{x_0 - \lambda}$ has been applied.
    By adapting the necessary and sufficient condition for the Weibull domain of attraction \citep{Gnedenko}, which is $F \in \mathcal{D}_-$ if and only if $x_0 < \infty$ and $1-F\left( \frac{x_0 x}{x + 1}\right) = C(x) x^{-\alpha}$ for $C$ a locally bounded function and slowly varying and $\alpha > 0 $, we obtain 
    $$\overline{F}_M(n) = \frac{x_0^{n+1}}{n!} \int_0^{\infty} \frac{\lambda^{n-\alpha}}{(\lambda+1)^{n+2}} C(\lambda) e^{-\frac{x_0 \lambda}{\lambda + 1}} d\lambda$$ and using the fact that the Beta function is such that
    $$\mathrm{B}(a,b) = \int_0^\infty \frac{t^{a-1}}{(t+1)^{a+b}}dt,$$
    a similar argument as in Theorem \ref{Thm:Ratio} provides that 
    \begin{align*}
        \overline{F}_M(n) &\sim \frac{x_0^{n+1}}{n!}\mathrm{B}(n-\alpha + 1, \alpha + 1) C(n) e^{-\frac{x_0 n}{n + 1}}\\
         &\sim \frac{x_0^{n+1}e^{-x_0}}{n!} C(n) \frac{\Gamma(n-\alpha+1)\Gamma(\alpha + 1)}{\Gamma(n+2)}\\
        &\sim \Gamma(\alpha + 1) C(n) n^{-\alpha} \left(\frac{x_0^{n+1}e^{-x_0}}{(n+1)!} \right).
    \end{align*}
\end{proof}
\noindent
Using the asymptotic behaviour in Theorem \ref{Thm:Asymp}, a similar result can be established for $P_M$.
\begin{corollary}
\label{cor:Asymp}
Let $F_M$ be a Poisson mixture with $\lambda$ distributed according to a cdf $F \in \mathcal{D}_-$.
Then the pmf $P_M$ is such that
$$P_M(n) \sim \Gamma(\alpha + 1) C(n) n^{-\alpha} \left( \frac{x_0^{n}}{n!} e^{-x_0} \right).$$    
\end{corollary}
\begin{proof}
    Since $P_M(n) = \overline{F}_M(n-1) - \overline{F}_M(n)$, then
    \begin{align*}
        \lim_{n \to \infty} \frac{P_M(n)}{\Gamma(\alpha + 1) C(n) n^{-\alpha} \left( \frac{x_0^{n}}{n!} e^{-x_0} \right)} &=  \lim_{n \to \infty} \frac{C(n-1) (n-1)^{-\alpha}}{C(n)n^{-\alpha}} -  \lim_{n \to \infty} \frac{x_0}{n+1} = 1.
    \end{align*}
\end{proof}
\noindent
This result provides a new perspective on why the limit in Theorem \ref{Thm:Ratio} converges to $0$ for a mixing distribution with a finite support.
Indeed, as previously mentioned, the Poisson distribution is an example such that the limit \eqref{Eqn:ShortTail} is satisfied.
From Theorem \ref{Thm:Asymp} and Corollary \ref{cor:Asymp}, $\overline{F}_M$ and $P_M$ behave like a Poisson distribution with mean $x_0$ multiplied by a regular varying function.
Intuitively, the mixing distribution does not put weight everywhere on $\mathbb{R}_+$, so the tail of $F_M$ cannot satisfy equation \eqref{Eqn:LongTail}.

\section{Numerical Study}
\label{Section:Numerical}
This section illustrates the practical implications of the theoretical results previously obtained. 
In particular, we highlight how the mixing distribution impacts the adjustment, how the statistical evaluation of tail distributions of count data may help to select a mixing distribution, and how the maxima of Poisson mixtures with finite mixing distribution behave asymptotically.

\subsection{Impact of mixing distribution choice on goodness of fit}

To illustrate how the tail behaviour of $\lambda$ affects the model adjustment, we simulated 100 samples of different Poisson mixtures with size $n=250$ using the (i) $\mathrm{Fr\Acute{e}chet}(\alpha, \beta)$, (ii) $\mathrm{lognormal}(\mu, \sigma)$, (iii) $\mathrm{Gamma}(\alpha, \beta)$, and (iv) $\mathrm{Uniform}(0, x_0)$ distributions on $\lambda$ with densities
\begin{enumerate}[(i)]
    \item $f(x) = \frac{\alpha}{x} \left( \frac{x}{\beta} \right)^{-\alpha} e^{-\left(\frac{x}{\beta}\right)^{-\alpha}}$, $\alpha > 0$, $\beta > 0$;
    \item $f(x) = \frac{1}{x \sigma \sqrt{2 \pi}} e^{-\frac{(\ln x - \mu)^2}{2 \sigma^2}}$ $\mu \in \mathbb{R}$, $\sigma > 0$;
    \item $f(x) = \frac{\beta^\alpha}{\Gamma(\alpha)} x^{\alpha-1} e^{-\beta x}$, $\alpha > 0$, $\beta > 0$;
    \item $f(x) = \frac{\mathds{1}_{(0,x_0)}(x)}{x_0}$,
\end{enumerate}
each one being a representative of four out of five type of mixing distributions we encountered.
Respectively, they are representative of elements in $\mathcal{D}_+$, $\mathcal{D}_0^\mathcal{H}$, $\mathcal{D}_0^\mathcal{E}$, and in $\mathcal{D}_-$.
Moreover, the parameter $\gamma$ from equation \eqref{Eqn:GEV} associated to (i) and (iv) are respectively $\gamma = 1/\alpha$, $\gamma = -1$ and $\gamma = 0$ for (ii) and (iii).
For each sample, the Poisson mixture is fitted with the same four distributions and the best model is kept using a Bayesian framework.
This is done using the language \texttt{R} \citep{R} and the \texttt{rstan} \citep{Stan} package to estimate the hyperparameters by MCMC.
The best model is then kept using the highest \textit{posterior} model probability.
Those probabilities are approximated using the bridge sampling computational technique \citep{Meng} and the dedicated \texttt{R} package \texttt{Bridgesampling} \citep{Bridge}.
All results are based on the following priors: a $\mathrm{Gamma}(1,1)$ distribution for positive parameters and a $\mathrm{Normal}(0,1)$ for real parameters. 
Moreover, we simulated for each sample four MCMCs with 10,000 iterations each in order to ensure reasonable convergence for parameter estimation and for the \textit{posterior} model probabilities.
Results are presented in Table \ref{tab:}. 

\begin{table}[!ht]
\centering
\begin{tabular}{ |c|c|c|c|c|c| }
 \hline
Mixing class & Mixing distribution &  Fréchet & Lognormal & Gamma & Uniform \\
\hline
& Fréchet(1,1) & \bf{89} & 11 & 0 & 0 \\
\cline{2-6} $\mathcal{D}_+$ & Fréchet(2,1) & \bf{80} & 18 & 2 & 0\\
\hline
& Lognormal(1,1) & 5 & \bf{89} & 6 & 0\\
\cline{2-6} $\mathcal{D}_0^\mathcal{H}$ & Lognormal(0,1) & 9 & \bf{69} & 23 & 0\\
\hline
& Gamma(2,1) & 1 & 22 & \bf{73} & 4\\
\cline{2-6} $\mathcal{D}_0^\mathcal{E}$ & Gamma(2,2) & 1 & 23 & \bf{54} & 22\\ 
\hline
& Uniform(0,10) & 0 & 0 & 26 & \bf{74}\\
\cline{2-6} $\mathcal{D}_-$ & Uniform(0,5) & 0 & 1 & 38 & \bf{61}\\
\hline
\end{tabular}
\caption{Selected model frequencies for each Poisson mixture simulation with the highest frequency in bold.}
\label{tab:}
\end{table}

The Poisson-Fréchet mixtures stood out the most since their tail is heavier than any other of the distributions. 
The only competing model seems to be the Poisson-lognormal which has a heavier tail than an exponential type distribution, but lighter than the Fréchet. 
The variance also influences what model is selected. 
Indeed, for example, the lognormal(0,1) has a lesser variance compared to the lognormal(1,1). 
In the former mixture, the Gamma seems to be able to compete against the lognormal, which is not the case for the latter.
Interestingly, the Fréchet mixing distribution is selected sparingly for lognormal data even when the variance gets larger.
This fact remains true for the rest of the table since the Fréchet distribution has a much heavier tail.
By Theorem \ref{Thm:Ratio}, we know that the Gamma distribution can get close to the Gumbel domain of attraction.
From Table \ref{tab:}, we see that the lognormal is a significant competitor for both simulations, which reflects the closeness to $\mathcal{D}_0$.
However, when the rate parameter is equal to $2$, the mean and variance decrease and the uniform becomes another chosen option.
This can be explained by the fact that $\frac{\overline{F}_M(n+1)}{\overline{F}_M(n)}$ is closer to $0$ when $n$ grows to infinity.
Finally, since the uniform has a finite tail, only the Gamma can compete and, again, larger the variance the less the Gamma is selected.
Based on each case, we see a diagonal effect from the heavier tail to the finite tail. 

\subsection{Identifying the domain of attraction}\label{Section_Strategy}
\label{SubSec:Excess}
In order to identify what domain of attraction a random variable belongs to, one can uses the peaks-over-threshold (POT) method \citep{Coles}.
This technique involves the distribution of the excesses defined by $Y - u|Y > u$, for a suitable choice of $u$. 
\cite{Pickands}, \cite{Balkema} showed that $Y$ belongs to a domain of attraction if and only if the distribution of the excesses converges weakly to a generalized Pareto distribution (GPD) as $u$ tends to the right endpoint of the distribution of $Y$. In such cases, the corresponding cdf is given by
\begin{equation}
 {H}_{\gamma, \sigma}(y) =
    \begin{cases}
     1- \left(1 + \gamma \frac{y}{\sigma} \right)^{-1/\gamma} & \text{if $\gamma \neq 0$,}\\
     1- \exp\left( -\frac{y}{\sigma}  \right) & \text{$\gamma = 0$,}
    \end{cases}       
\end{equation}
with support $\mathbb{R}_+$ if $\gamma \geq 0$ or $\left[0; -\frac{\sigma}{\gamma} \right]$ if $\gamma < 0$, where $\gamma \in \mathbb{R}$ and $\sigma > 0$ are respectively shape and scale parameters. 
Moreover, the $\gamma$ parameter is the same as in equation \eqref{Eqn:GEV}. 
Therefore, fitting a GPD to the excesses of a sample can inform us on the domain of attraction the underlying distribution belongs to.
Better yet, excesses of count data can inform us whether or not a Poisson mixture distribution belongs to a known domain of attraction and, if so, which one.
Therefore, analyzing the discrete excesses can indicate what type of mixing distribution generates the Poisson mixture.
Indeed, by Theorem \ref{Thm:Ratio}, if the discrete excesses belong to a domain of attraction, then a mixing distribution $F$ should be in $\mathcal{D}_+ \cup \mathcal{D}_0^\mathcal{H}$.
Otherwise, $F$ should either have an exponential or finite tail.

From a practical point of view, the study of discrete excesses may justify a choice of model.
For example, one may hesitate between adjusting a Poisson-lognormal or a negative binomial for their count data.
In order to study how useful the discrete excesses can be, various Poisson mixtures have been simulated.
Here, we fixed the sample size to $n=1000$, the threshold $u$ to be the 95th or 97.5th empirical quantiles, and simulated $1000$ samples for each mixing distribution.
For each sample, the discrete excesses are extracted, and the \texttt{evd} R package \citep{EVD} is used to estimate the GPD parameters by maximum likelihood.
Based on these estimations, the modified Anderson Darling test for the goodness-of-fit is applied.
Finally, for the samples such that the GPD appears to be adequate, we test $H_0:\gamma = 0$ versus $H_1:\gamma\neq 0$.
To do so, we fit these two models, evaluate the corresponding log likelihoods $\mathcal{L}_1$ and $\mathcal{L}_0$, and conclude with the deviance statistic $D = 2\left(\mathcal{L}_1 - \mathcal{L}_0 \right)$ which follows approximately a $\chi_1^2$ distribution under suitable conditions \citep{Coles}.
Results are presented in Table \ref{Tab:Sim}.
\begin{table}[h]
\centering
\begin{tabular}{|c|c|c|c|c|c|}
\hline
 Mixing Class & Mixing distribution & u & Average number of access & GPD Rejection & Test $\gamma = 0$ not rejected\\
 \hline
  & & 95 & 48.727 & 0.069 & 0.014 \\
  & Fréchet(1,1) & 97.5 & 24.685 & 0.051  & 0.158 \\
\cline{2-6} $\mathcal{D}_+$  & & 95 & 41.915 & 0.777 & 0.170 \\
  & Fréchet(2,1) & 97.5 & 21.746 & 0.177 & 0.615\\ 
 \hline
  & & 95 & 46.750 & 0.126 & 0.720 \\
  & Lognormal(1,1) & 97.5 & 23.644 & 0.037  & 0.845 \\
\cline{2-6} $\mathcal{D}_0^\mathcal{H}$  & & 95 & 41.914 & 0.697 & 0.257\\
  & Lognormal(0,1) & 97.5 & 21.685 & 0.142 & 0.790 \\ 
 \hline
  & & 95 & 36.200 & 0.704 & 0.045 \\
  & Gamma(2,1) & 97.5 & 18.876 & 0.245  & 0.502\\
\cline{2-6} $\mathcal{D}_0^\mathcal{E}$  & & 95 & 38.015 & 0.833 & 0.052\\
  & Gamma(2,2) & 97.5 & 16.988 & 0.392 & 0.311\\ 
 \hline
  & & 95 & 39.124 & 0.641 & 0.028 \\
  & Uniform(0,10) & 97.5 & 18.999 & 0.296  & 0.390 \\
\cline{2-6} $\mathcal{D}_-$  & & 95 & 35.161 & 0.679 & 0.059 \\
  & Uniform(0,5) & 97.5 & 18.087 & 0.369 & 0.255 \\ 
 \hline
\end{tabular}
\caption{Average number of excesses, rejection rate for the GPD and non-rejection rate of $H_0: \gamma = 0$ for the simulations with $n = 1000$ and $u = 95\mathrm{th}$ or $97.5\mathrm{th}$ empirical quantile.}
\label{Tab:Sim}
\end{table} 

Firstly, we notice that even if the Fréchet and lognormal distributions are in $\mathcal{D}_+$ and $\mathcal{D}_0^\mathcal{H}$ respectively, the Fréchet(2,1) and lognormal(0,1) cases lead to a high rejection rate for the 95th quantile threshold.
However, when both cases are simulated with a threshold $u$ equal to the 97.5th quantile, the rate of GPD rejection diminishes.
Therefore, it seems that the threshold choice has a great impact.
Moreover, when $u$ is the 97.5th quantile, the estimation of $\gamma$ is not significantly different to $0$ for 79 \% of the samples of the lognormal(0,1).
However, 61.5 \% of the samples of the Fréchet are also significantly null.
Secondly, as noted by \cite{Hitz}, the discrete excesses need a certain amount of variability in order to have a smooth adjustment to the GPD.  
Since the lognormal(1,1) has a greater variance and the Fréchet(1,1) doesn't have a finite expectation, this explains why these cases are well adjusted to the GPD.
Finally, both Gamma and uniform cases have GPD rejection rates as expected.
Interestingly, the uniform distribution is rejected at a lesser rate then the Gamma.
Again, this can be explained by the greater variance for the uniform than the Gamma simulations.\medskip

Also, the Gamma(2,1) leads to a lower rate of rejection than the Gamma(2,2), which is reasonable since the former is closer to $\mathcal{D}_0$ than the latter by Theorem \ref{Thm:Ratio}.
Indeed, if the limit in Theorem \ref{Thm:Ratio} $(1+\beta)^{-1}$ approaches $0$, the GPD rejection rate for the Poisson mixtures should increase.
Inversely, the rejection rate should decrease when $(1+\beta)^{-1}$ approaches $1$.
To further analyze this, we simulated Poisson mixtures with a Gamma($2$, $\beta$) mixing density and let the parameter $\beta$ vary from $0.1$ to $8$, the quantity $(1+\beta)^{-1}$ thus varying between $1/9$ and $10/11$.
For each value of $\beta$, we simulate 500 samples of size $n=1000$ from the Poisson mixture, fix the threshold $u$ to the 95th empirical quantile, and calculate the proportion of samples where the GPD is rejected with type I error $\alpha = 0.05$.
Results are presented in Figure \ref{Fig:reject}.
We can see that indeed the proportion decreases when $(1+\beta)^{-1}$ moves towards $1$.
Between $0$ and $0.5$, the rejection proportion oscillates between $0.5$ and $1$.
This can be explained by the fact that the number of discrete excesses also oscillates when $\beta$ increases, which affects the power of the test.\medskip

To adjust for the problems related to the discreteness of the excesses, it would be interesting to transform them into continuous variables.
As demonstrated by \cite{Shimura}, a Poisson mixture with $F \in \mathcal{D}_0^\mathcal{H}$ is a random variable that originates from an unique continuous distribution in $\mathcal{D}_0$ that has been discretized.
If one can identify such a continuous distribution associated to the discrete excesses when the GPD is rejected, then it would be reasonable to use an exponential tail mixing distribution.
A jittering technique consiting of adding random noise to data has been proposed for different discrete contexts \citep{Nagler, Coeurjolly-Rousseau}.
A plausible approach would be a jittering for the GPD test in order to adequately identify the type of mixing distribution associated to the discrete excesses.

\begin{figure}[h]
    \centering
    \includegraphics[width=\textwidth]{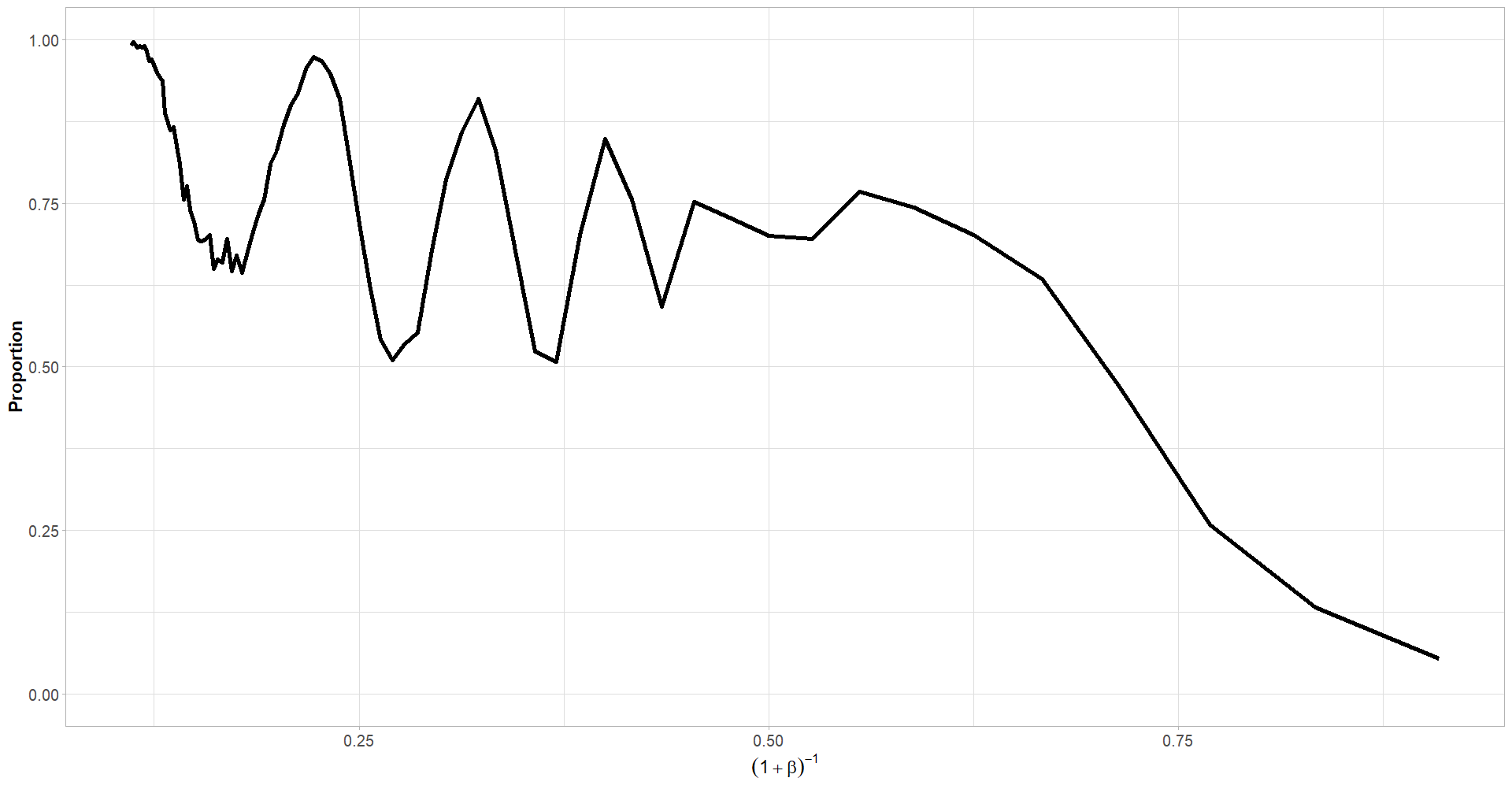}
    \caption{Proportion of Gamma($2$, $\beta$) Poisson mixture samples (size $n=1000$) where the GPD has been rejected ($\alpha = 0.05$) for the excesses ($u =$ 95th empirical quantile) as a function of $(1+\beta)^{-1}$.}
    \label{Fig:reject}
\end{figure}

\subsection{Maxima for Poisson mixtures with finite tail mixing distribution}
By Theorem \ref{Thm:Ratio}, if $F$ has bounded support $(0, x_0)$, then the Poisson mixture is short tailed, i.e. $\frac{\overline{F_M}(n+1)}{\overline{F_M}(n)}\to 0$ as $n \to \infty$.
Therefore, according to \cite{Anderson}, there exists a sequence of integers $I_n$ such that equation \eqref{Eqn:MaxSeq} is satisfied.
Moreover, by Corollary \ref{cor:Asymp}, the pmf $P_M$ asymptotically behaves like a Poisson distribution and, as mentioned, the Poisson is the primary example where its maximum oscillates between two integers.
\cite{Kimber} and \cite{Briggs} study how the sequence $I_n$ can be approximated for the Poisson distribution and showed that it grows slowly when $n \to \infty$.
Since $P_M$ behaves like the Poisson when $F$ is in $\mathcal{D}_-$, the sequence $I_n$ should also grow slowly.
To visualise this behaviour, we simulated Poisson mixtures with $\lambda \ \sim x_0\mathrm{Beta}(\alpha, \beta)$.
We fixed $\alpha = 2$, $x_0 = 5$, and for $n \in \{10, 10^2, 10^3, 10^4\}$, we simulated 10000 samples of $F_M$ with size $n$ and recorded the maximum for each sample.
With these maxima, we calculated the empirical probabilities, and repeated for $\beta \in \{1/4, 1/2, 1, 2\}$.
Figure \ref{Fig:maximum} reports on the empirical and theoretical pmf of the simulations and the maxima of $n$ Poisson variables with mean $x_0$ respectively.
\begin{figure}[h]
    \centering
    \includegraphics[width=\textwidth]{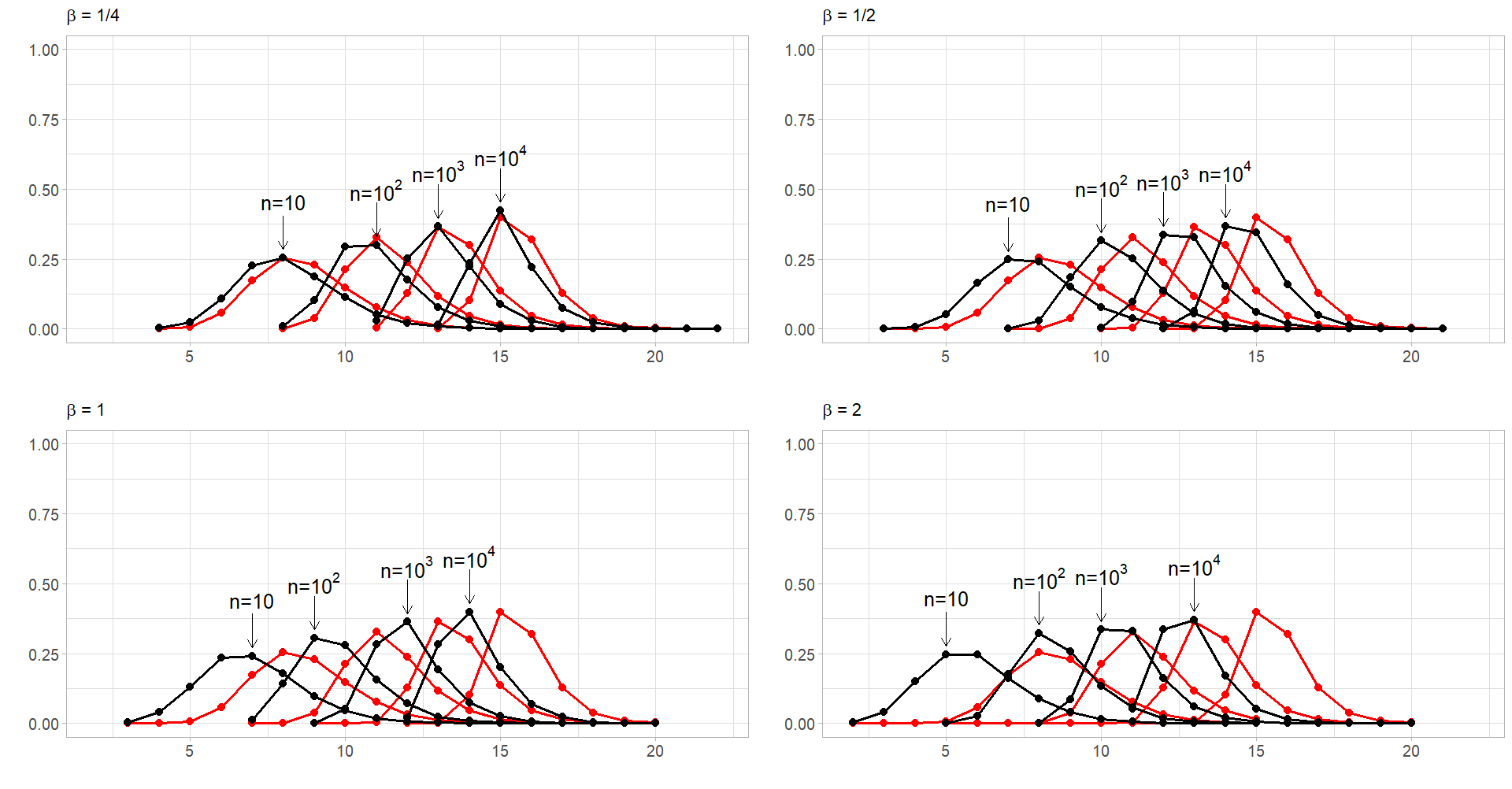}
    \caption{Maximum distributions of Poisson mixture with $\lambda \sim x_0\mathrm{Beta}(2,\beta)$ (black) and $\mathrm{Poisson}(x_0)$ (red) with $x_0 = 5$, $\beta \in \{1/4, 1/2, 1, 2\}$ and $n \in \{10, 10^2, 10^3, 10^4\}$.}
    \label{Fig:maximum}
\end{figure}
Interestingly, the greater $\beta$ becomes, the slower the sequence $I_n$ increases.
Indeed, when $\beta = 1/4$, the probability distribution of the maxima looks similar to that of a Poisson($x_0$).
For $\beta = 2$, the distribution for the Poisson mixture drastically shifts to the left.
This can be explained using Corollary \ref{cor:Asymp}.
Indeed, we can show that $P_M$ here is such that $$P_M(n) \sim \frac{\Gamma(\alpha + \beta)}{\Gamma(\alpha)} n^{-\beta} \left( \frac{x_0^n e^{-x_0}}{n!}\right),$$
and when $\beta$ approaches $0$, then only the pmf of the Poisson($x_0$) remains.
From another point of view, the density of the $x_0\mathrm{Beta}(\alpha, \beta)$ approaches a Dirac on $x_0$, so the Poisson mixture approaches a simple Poisson distribution.

\section{Conclusion and perspectives}
Overdispersed count data are commonly observed in many applied fields and Poisson mixtures are appealing to model such data. 
However, the choice of the appropriate mixing distribution is a difficult task relying mainly on empirical approaches related to modelers subjectivity or on intensive computational techniques combined with goodness-of-fit test or information criteria. 
In this paper, we showed that such a choice should respect the relation between the tail behaviour of $\lambda$ and the discrete data.
Indeed, if a distribution $F$ is in the Fréchet domain of attraction or satisfies the Gumbel hazard condition given by Definition \ref{Def:GumbelCond}, then the discrete data should be in the same domain of attraction.
Otherwise, an exponential or finites tail should be chosen.
Moreover, Theorem \ref{Thm:Ratio} established that Poisson mixtures with $F \in \mathcal{D}_0$ need to be separated into three subsets: $\mathcal{D}_0^\mathcal{E}$, $\mathcal{D}_0^\mathcal{H}$ and $\mathcal{D}_0^\mathcal{F}$.  
Both subsets $\mathcal{D}_0^\mathcal{E}$ and $\mathcal{D}_0^\mathcal{H}$ have distributions belonging to a larger subset named Weibull tail \citep{Gardes}.
It would be interesting to generalize Theorem \ref{Thm:Ratio} with this familly of mixing distributions.\smallskip

To identify whether the data distribution comes from a domain of attraction or not, we have studied the discrete excesses and their adjustment by the GPD. 
Some difficulties occurred due to the discrete nature of the data.
Solutions that could be explored are the use of techniques like the jittering or the use of discrete analogues of the GPD like the discrete generalized Pareto or the generalized Zipf distribution presented in \cite{Hitz}.
These approaches should help identify whether $\lambda$ has a exponential tail or not.
However, one could think about testing if $\lambda$ has a bounded support.
Based on Theorem \ref{Thm:Asymp} and Corollary \ref{cor:Asymp}, the Poisson mixture with a finite mixing distribution should behave similarly to a Poisson with mean $x_0$.
Testing whether $F$ has a finite tail or not based on these results is a promising avenue.\smallskip

In the field of extreme value theory, our Theorem \ref{Thm:Asymp} and the result of \cite{Willmot} in Proposition \ref{Prop:Willmot} may provide an approach to finding normalizing sequences such that the Poisson mixture belongs to a domain of attraction.
Indeed, \cite{Anderson2} showed that if the Poisson's mean $\lambda$ depends on the sample size and increases with a certain rate, then it is possible to find normalizing sequences $a_n$ and $b_n$ such that the distribution is in the Gumbel domain of attraction.
If $\lambda$ does not depend on the sample size, then no such sequence can be found.
A similar result has been proved by \cite{Nadarajah} for the negative binomial when $\alpha$ is fixed and $\beta$ approaches $0$.
Since Theorem \ref{Thm:Asymp} and Proposition \ref{Prop:Willmot} showed that Poisson mixtures with finite or exponential tail mixing distribution resemble the Poisson or the negative binomial respectively, one could exploit these asymptotic properties to generalize the results of \cite{Anderson2} and \cite{Nadarajah} with various Poisson mixtures like the Poisson-inverse-Gaussian  or Poisson-Beta.
Similarly, generalizing the results of \cite{Kimber} and \cite{Briggs} concerning the sequence $I_n$ for the maxima of Poisson random variables should also be explored.
\section*{Acknowledgments}

This research was supported by the GAMBAS project funded by the French National Research Agency (ANR-18-CE02-0025) and the French national programme LEFE/INSU.

\bibliography{Biblio}
\end{document}